\documentclass[12pt]{amsart}
\usepackage{enumerate}
\usepackage[hyperindex,breaklinks,colorlinks,citecolor=blue,pagebackref]{hyperref}
\usepackage[latin1]{inputenc}
\usepackage{amsmath, amsthm, amsfonts, amssymb}
\usepackage{mathrsfs}
\usepackage{graphicx}
\usepackage{latexsym}
\usepackage{fullpage}
\usepackage[all]{xy}
\usepackage{color}
\usepackage{stmaryrd}
\usepackage[mathscr]{euscript}

\newtheorem{theorem}{Theorem}[section]
\newtheorem{lemma}[theorem]{Lemma}
\newtheorem{proposition}[theorem]{Proposition}
\newtheorem{corollary}[theorem]{Corollary}
\newtheorem{question}[theorem]{Question}

\theoremstyle{remark}

\newtheorem{rmks}{Remarks}[section]
\newtheorem*{Examples}{Examples}

\newcommand{\MLR}{\mathsf{MLR}}

\newcommand{\dom}{\text{dom}}

\newcommand{\jump}{\emptyset'}

\newcommand{\llb}{\llbracket}
\newcommand{\rrb}{\rrbracket}

\definecolor{purple}{rgb}{.9,0.2,.9}

\newcommand{\cs}{2^\omega}

\newcommand{\uh}{{\upharpoonright}}
\renewcommand{\phi}{\varphi}
\newcommand{\str}{2^{<\omega}}

\newcommand{\halts}{{\downarrow}}

\usepackage{xcolor}	
	\usepackage{soul}

\title{Continuous Randomness via Transformations of 2-Random Sequences}
\author{Christopher P.\ Porter}
\date{} 

\begin{document}

\begin{abstract}
Reimann and Slaman initiated the study of sequences that are Martin-L\"of random with respect to a continuous measure, establishing fundamental facts about NCR, the collection of sequences that are not Martin-L\"of random with respect to any continuous measure.  In the case of sequences that are random with respect to a computable, continuous measure, the picture is fairly well-understood:  such sequences are truth-table equivalent to a Martin-L\"of random sequence.  However, given a sequence that is random with respect to a continuous measure but not with respect to any computable measure, we can ask:  how close to effective is the measure with respect to which it is continuously random?  

In this study, we take up this question by examining various transformations of 2-random sequences (sequences that are Martin-L\"of random relative to the halting set $\emptyset'$) to establish several results on sequences that are continuously random with respect to a measure that is computable in  $\emptyset'$.  In particular, we show that (i) every noncomputable sequence that is computable from a 2-random sequence is Martin-L\"of  random with respect to a continuous, $\emptyset'$-computable measure and (ii) the Turing jump of every 2-random sequence is Martin-L\"of random with respect to a continuous, $\emptyset'$-computable measure.  From these results, we obtain examples of sequences that are not proper, i.e., not random with respect to any computable measure, but are random with respect to a continuous, $\emptyset'$-computable measure.  Lastly, we consider the behavior of 2-randomness under a wider class of effective operators (c.e.\ operators, pseudojump operators, and operators defined in terms of pseudojump inversion), showing that these too yield sequences that are Martin-L\"of random with respect to a continuous, $\emptyset'$-computable measure.
\end{abstract}

\maketitle

\section{Introduction}\label{sec-intro}

The study of algorithmically random sequences with respect to noncomputable measures was initiated by Levin in \cite{Lev76} and was significantly advanced by Reimann and Slaman \cite{ReiSla15, ReiSla21} and Day and Miller \cite{DayMil13}.  Reimann and Slaman focused in particular on Martin-L\"of randomness with respect to continuous measures, showing in particular that  all but countably many sequences are Martin-L\"of random with respect to a continuous measure on Cantor space.  This contrasts significantly with the case of randomness with respect to a \emph{computable} continuous measure, as every sequence that is Martin-L\"of random with respect to such a measure is truth-table equivalent to an unbiased Martin-L\"of random sequence (hereafter, ``random" will be short for ``Martin-L\"of random") .

The present study is motivated by the question:  given a sequence that is random with respect to a continuous measure, how close to effective is the continuous measure with respect to which it is random?  Clearly being random with respect to a computable, continuous measure is the best we can hope for.  But what if our sequence is not \emph{proper}, that is, not random with respect to a computable measure?  The aim of this brief study is to move one level up in the arithmetical hierarchy to find such sequences that are random with respect to a  measure that  is computable in the halting set $\emptyset'$.

We will focus on two general approaches to generating nonproper sequences that are random with respect to a $\emptyset'$-computable measure.  First, in Section \ref{sec-2ran} we establish a  useful connection between Turing reductions from 2-random sequences and randomness with respect to a $\emptyset'$-computable measure. In particular, we will prove that every noncomputable sequence that can be computed by a 2-random sequence is random with respect to a continuous, $\emptyset'$-computable measure, which improves a result due to Reimann and Slaman \cite{ReiSla21} (who showed that every noncomputable sequence computable from a 3-random sequence is random with respect to a continuous measure).  From this result, we will be able to obtain examples of nonproper sequences that are random with respect to a continuous, $\emptyset'$-computable measure.

In Section \ref{sec-operators}, we will then consider the behavior of 2-random sequences under a broader class of effective operators beyond Turing functionals, including the Turing jump and pseudojump operators.  In particular, we will show that the jump of a 2-random sequence is a nonproper sequence that is random with respect to a continuous, $\emptyset'$-computable measure.

For related work, see Hirschfeldt and Terwijn \cite{HirTer08} on $\Delta^0_2$ measures. Demuth also studied $\emptyset'$-computable measures in \cite{Dem88a} and \cite{Dem88b}.  Cenzer and Porter \cite{CenPor18} also studied  several notions of randomness for members of $\Pi^0_1$ classes that are given in terms of certain $\emptyset'$-computable measures.  See also \cite{LiRei19} for recent work on sequences that are not random with respect to any continuous measure.

Before turning to our main results, we briefly review a few concepts and fix our notation. In what follows, $\lambda$ stands for the Lebesgue measure. For binary strings $\sigma,\tau\in\str$, we use the notation $\sigma\preceq\tau$ to indicate that $\sigma$ is an initial segment of $\tau$.  We similarly define $\sigma\prec X$ for $\sigma\in\str$ and $X\in\cs$.  Moreover, we write the concatenation of $\sigma$ and $\tau$ as $\sigma^\frown\tau$.  We write the empty string as $\epsilon$.  Given $\sigma\in\str$, $\llb\sigma\rrb=\{X\in\cs:\sigma\prec X\}$ is the \emph{cylinder set} determined by $\sigma$.

A map $\Phi:\subseteq \cs \to \cs$ is a Turing functional if there is an oracle Turing machine that when given
$X \in \dom(\Phi)$ as an oracle computes the characteristic function of some $Y\in\cs$; in this case, we write $\Phi(X){\downarrow}=Y$.  For $k\in\omega$, we will write $\Phi(X;k)$ to be the output the computation on input $k$.  Of course, there may be some $k$ such that $\Phi(X;k)$ is undefined; in this case, we will consider $\Phi(X)$ to be undefined.  We can define the domain of $\Phi$ to be $\mathrm{dom}(\Phi)=\{X\in\cs\colon \Phi(X){\downarrow}\}$. We can also relativize any such  functional to some $Z\in\cs$ to obtain a $Z$-computable functional.

Recall that a measure on $\cs$ is determined by the values it assigns to the cylinder sets.  A measure $\mu$ on $\cs$ is computable if the value $\mu(\llb\sigma\rrb)$ is a
computable real number uniformly in $\sigma\in 2^{<\omega}$\!.  Similarly, for $Z\in\cs$ we can define a $Z$-computable measure $\mu$ by using $Z$ as an oracle to compute the values $\mu(\llb\sigma\rrb)$ for $\sigma\in\str$.

If $\mu$ is a $Z$-computable measure on $\cs$ and
$\Phi\colon \subseteq \cs \to \cs$ is a $Z$-computable  functional defined on
a set of $\mu$-measure one, then the \emph{pushforward measure}
$\mu_\Phi$ defined by setting
\begin{displaymath}
  \mu_\Phi(\llb\sigma\rrb)=\mu(\Phi^{-1}(\llb\sigma\rrb))
\end{displaymath}
for each $\sigma\in \str$ is a $Z$-computable measure.

 Recall that for a fixed computable measure $\mu$ on $\cs$\! and $Z\in \cs$\!, a 
 \emph{$\mu$-Martin-L\"of test relative to $Z$} (or simply a \emph{$\mu$-test relative to $Z$}) is a uniformly
    $\Sigma^{0}_{1}[Z]$ sequence $(U_{i})_{i\in\omega}$ of subsets of
    $\cs$ with $\mu (U_{n}) \leq 2^{-i}$\! for every $i\in\omega$. $X\in \cs$ passes such a test $(U_{i})_{i\in\omega}$
    if $X\notin \bigcap_{i\in\omega} U_{i}$ and $X$ is $\mu$-Martin-L\"of random relative to
    $Z$ if $X$ passes every $\mu$-Martin-L\"of test relative to
    $Z$. The set of all such sequences $X$ is denoted by $\MLR_{\mu}^{Z}$.  For each choice of $\mu$ and $Z$ as above, there is a single, \emph{universal}, $\mu$-test relative to $Z$,
$(U_{i}^Z)_{i\in \omega}$ such that $X \in \MLR_\mu^Z$ if and only if $X$ passes $(U_{i}^Z)_{i\in \omega}$.   Lastly, we say that a sequence is \emph{proper} if it is Martin-L\"of random with respect to a computable measure.

In the case that $\mu$ is a noncomputable measure, we have to be careful in defining randomness with respect to $\mu$.  In particular, we must relativize our tests to some sequence $R\in\cs$ that encodes our measure $\mu$; such a sequence is a called a \emph{representation} of $\mu$.  Specific details about representations of measures can be found, for instance in \cite{DayMil13} or \cite{ReiSla21}.  For our purposes, we do not need the full machinery of the representation of measures.  Whereas in the general approach to randomness with respect to a noncomputable measure $\mu$, a sequence is $\mu$-Martin-L\"of random if it is Martin-L\"of random with respect to \emph{some} representation $R_\mu\in\cs$ of $\mu$,  in the present study, we only need show that a given sequence is random with respect to a $\emptyset'$-computable measure, so it suffices to consider our tests relative to the oracle $\emptyset'$.

An \emph{atom} of a measure $\mu$ is a sequence $A\in\cs$ such that $\mu(\{A\})>0$.  A measure is \emph{continuous} if it has no atoms; otherwise it is \emph{atomic}.  A routine relativization of a result due to Kautz \cite{Kau91} yields the following.
\begin{lemma}\label{lem-atom}
For $Z\in\cs$, $A\in\cs$ is an atom of some $Z$-computable measure if and only if $A\leq_TZ$.
\end{lemma}
\noindent One can further show that if $A\leq_TZ$ and $A$ is Martin-L\"of random with respect to a $Z$-computable measure $\mu$, then $A$ must be an atom of $\mu$.

We  will make use of a pair of results concerning the interaction between Turing functionals and Martin-L\"of randomness:

\begin{theorem}\label{thm-pres}
  For $Z\in\cs$, let $\Phi\colon \cs \to \cs$ be a $Z$-computable functional and let $\mu$ be a $Z$-computable measure satisfying $\mu(\dom(\Phi))=1$.
  \begin{itemize}
    \item[(i)] (Randomness preservation \cite{ZvoLev70}) If $X \in \MLR_{\mu}^Z$ then
    $\Phi(X) \in \MLR_{\mu_\Phi}^Z$.
    \item[(ii)] (No
  randomness from nonrandomness \cite{She86}) If $Y \in \MLR_{\mu_\Phi}^Z$, then there is some
    $X\in \MLR_{\mu}^Z$ such that $\Phi(X)=Y$.
  \end{itemize}
  \label{thm:PoR-and-NReN}
\end{theorem}

Our study is primarily concerned with 2-randomness, that is, Martin-L\"of randomness relative to the halting set $\jump$, but we will make use of the fact that every 2-random sequence is weakly 2-random.  Recall that a sequence $X\in\cs$ is weakly 2-random if for every $\Pi^0_2$ class $P$ with $\lambda(P)=0$, we have $X\notin P$. Two useful facts about every weakly 2-random sequence (and hence every 2-random sequence) that we will use are as follows.  First, if $X$ is weakly 2-random and $\Phi$ is a Turing functional with $X\in\dom(\Phi)$, then $\lambda(\dom(\Phi))>0$ (since the domain of a Turing functional is a $\Pi^0_2$ class).  Second, as shown by Downey, Nies, Weber, and Yu \cite{DowNieWeb06}, $X\in\cs$ is weakly 2-random if and only if $X$ is Martin-L\"of random and $X$ forms a minimal pair with $\emptyset'$ in the Turing degrees (that is, $X$ does not compute any noncomputable $\Delta^0_2$ sets).

For more background on algorithmic randomness, see \cite{Nie09}, \cite{DowHir10}, \cite{SheUspVer17}, or the surveys contained in \cite{FraPor20}.

\section{Computing from 2-random sequences}\label{sec-2ran}

As noted in the previous section, Reimann and Slaman \cite{ReiSla21} proved that every noncomputable sequence below a 3-random sequence (i.e., a sequence that is Martin-L\"of random with respect to $\emptyset''$) is random with respect to a continuous measure.
 We improve this result as follows.
 
 \begin{theorem}\label{thm-below2r}
Every noncomputable sequence Turing below a 2-random sequence is Martin-L\"of random with respect to a continuous, $\emptyset'$-computable measure.
 \end{theorem}
 
Note that we are not claiming that a noncomputable sequence below a 2-random sequence is 2-random with respect to a $\emptyset'$-computable measure.  In the terminology laid out, for instance, in \cite{ReiSla21}, 2-randomness with respect to a noncomputable measure requires that we consider the jump of a representation of our measure; in our case, we only need $\emptyset'$ as an oracle.
 
Our proof of Theorem \ref{thm-below2r} relies on a combination of several tools.  First, we will make use of a class of Turing functionals, first isolated by Barmpalias, Day, and Lewis-Pye \cite{BarDayLew14} in their study of the typical Turing degree, which are referred to as \emph{special} Turing functionals.  Here a Turing functional is special if its range does not include any computable sequences.  The key result we will use is the following (the result we draw upon in \cite{BarDayLew14} is slightly more general):
 
 
%
%

\begin{lemma}[Barmpalias, Day, Lewis-Pye \cite{BarDayLew14}]\label{lem-special}
 If $X$ is a 2-random sequence and $Y$ is a noncomputable sequence such that $\Phi(X)=Y$ for some Turing functional $\Phi$, then $\Psi(X)=Y$ for some special Turing functional $\Psi$.
\end{lemma}

Next, we will use a classical result due to Sacks.

\begin{theorem}[Sacks \cite{Sac63}]
For $X\in\cs$ and a Turing functional $\Phi$,  if $\lambda(\Phi^{-1}(X))>0$, then $X$ is computable.
\end{theorem}

Lastly, we will use the following lemma.

\begin{lemma}[Functional Extension Lemma]\label{lem-fe}
For $Z\in\cs$, let $\Phi^Z$ be a $Z$-computable functional that is total on a $\Pi^0_1[Z]$ class $P$.  Then there is a total $Z$-computable functional $\Psi^Z$ that agrees with $\Phi^Z$ on $P$.  Moreover, we can define $\Psi^Z$ so that if $X\in\cs\setminus P$, $\Psi^Z(X)$ is a finite modification of $X$.
\end{lemma}

\begin{proof}
Given $Z\in\cs$, $\Phi^Z$, and $P$ as above, we define $\Psi^Z$ in terms of a $Z$-computable approximation of $P$ given by clopen sets $(P_s)_{s\in\omega}$ (where $P_{s+1}\subseteq P_s$ for every $s\in\omega$ and $P=\bigcap_{s\in\omega}P_s$).  For $X\in\cs$, we set
\[
\Psi^Z(X;n)=
	\left\{
		\begin{array}{ll}
			\Phi^Z(X; n) & \mbox{if } (\exists s\geq n)( \Phi^Z_s(X\uh s; n)\halts\;\&\;X\in P_s) \\
			X(n) & \mbox{otherwise}
		\end{array}.
	\right.
\]
Note that if $X\in P$, since $\Phi^Z$ is total on $P$, such a stage $s$ exists.  In this case, $\Psi^Z(X; n)$ agrees with $\Phi^Z(X; n)$.  
Moreover, given $X\in\cs\setminus P$, there is some $m$ such that for all $s\geq m$, $X\notin P_s$ (and we can $Z$-computably detect when this occurs).  Thus we will have $\Psi^Z(X;n)=X(n)$ for all $n\geq m$, as desired.

\end{proof}

We now turn to the proof of Theorem \ref{thm-below2r}.  Here we draw on a technique used to show that every 2-random sequence $X$ is generalized low, i.e., $X\oplus \jump\equiv_T X'$ (a generalization of which is due to Kautz \cite{Kau91}; see also the proof of \cite[Lemma 4.4]{Sim07}).

\bigskip

\begin{proof}[Proof of Theorem \ref{thm-below2r}]

Given a noncomputable sequence $Y\leq_T X$ where $X$ is 2-random, let  $\Phi$ be a Turing functional that witnesses this reduction. By Lemma \ref{lem-special}, we can assume that $\Phi$ is a special Turing functional.  By our discussion in Section \ref{sec-intro}, since $X$ is weakly 2-random, we have $\lambda(\dom(\Phi))>0$.

Write $\dom(\Phi)=\bigcap_{n\in\omega}S_n$, where $S_n=\{Z\in\cs\colon \Phi(Z;n)\halts\}$.  We define a function $f:\omega\rightarrow\omega$ such that for each $n\in\omega$, $f(n)$ is the least stage $s$ such that $\lambda(S_n\setminus S_{n,s})\leq 2^{-n}$.  Clearly $f\leq_T\jump$.  Then we set $V_n=S_n\setminus S_{n,f(n)}$ for each $n\in\omega$, and we further set $W_n=\bigcup_{i>n}V_i$. Since $(W_n)_{n\in\omega}$ is uniformly $\Sigma^0_1[\jump]$ and $\lambda(W_n)\leq 2^{-n}$ for every $n\in\omega$, $(W_n)_{n\in\omega}$ is a Martin-L\"of test relative to $\jump$.

Since $X\in\dom(\Phi)$, $X\in S_n$ for every $n\in\omega$.  Moreover, since $X$ is 2-random, $X\notin W_j$ for some $j\in\omega$, which implies that $X\notin V_i$ for every $i>j$.  It follows that for every $i>j$, $X\notin S_i\setminus S_{i,f(i)}$, so that $X\in S_{i,f(i)}$ for all but finitely many $i\in\omega$.  Thus there is some $k\in\omega$ such that $X\in S_{i,f(i)+k}$ for all $i\in\omega$. Then $S=\bigcap_{i\in\omega}S_{i,f(i)+k}$ is a $\Pi^0_1[\jump]$ subclass of $\dom(\Phi)$.

%
%
%

Again using the fact that $X$ is 2-random,  for the universal $\emptyset'$-Martin-L\"of test $(U^{\emptyset'}_i)_{i\in\omega}$, there is some  $j\in\omega$ such that $X\in S\cap (\cs\setminus U^{\emptyset'}_j)$.  Moreover, since $X$ is not contained in any $\Pi^0_1[\emptyset']$-classes of measure 0, it follows that $\lambda(S\cap (\cs\setminus U^{\emptyset'}_j))>0$.  We set $P=S\cap (\cs\setminus U^{\emptyset'}_j)$, a $\Pi^0_1[\emptyset']$ subset of $\dom(\Phi)$ of positive measure that contains only 2-random sequences.

By the Functional Extension Lemma  (Lemma \ref{lem-fe}) applied to the case that $Z=\jump$, since $\Phi$ is total on $P$, there is a $\emptyset'$-computable functional $\Psi^{\emptyset'}$ that agrees with $\Phi$ on $P$.  Moreover, we can further assume in the case that $Y\notin P$, $\Psi^{\emptyset'}(Y)$ is a finite modification of $Y$. Hereafter, we will write $\Psi^{\jump}$ as $\Psi$.

Since $X\in P$ is 2-random, it follows by randomness preservation (Theorem \ref{thm-pres}(i)) that $\Psi(X)$ is 2-random with respect to $\lambda_\Psi$, the $\emptyset'$-computable measure induced by $\Psi$.  We claim that $\lambda_\Psi$ is continuous. Suppose for the sake of contradiction that $\lambda_\Psi$ is not continuous.  Then there is some $A\in\cs$ such that $\lambda_\Psi(\{A\})=\lambda(\Psi^{-1}(\{A\}))>0$.  We have two cases to consider.
%
%

\medskip

\emph{Case 1}:  $\lambda(\Psi^{-1}(\{A\})\cap P)>0$.  Since $\Phi$ agrees with $\Psi$ on $P$, $\lambda(\Phi^{-1}(\{A\})\cap P)>0$ and hence $A$ is computable by Sacks' Theorem.  However, since $P$ only contains 2-random sequences and $\Phi$ was chosen to be special, no sequence in $P$ computes a computable sequence via $\Phi$.  So this case is impossible.

\medskip

\emph{Case 2}:  $\lambda(\Psi^{-1}(\{A\})\setminus P)>0$.  Then by the definition of $\Psi$, each sequence in $\Psi^{-1}(\{A\})\setminus P$ is a finite modification of $A$.  As the set of finite modifications of a fixed sequence has Lebesgue measure 0, it follows that $\lambda(\Psi^{-1}(\{A\})\cap P)=0$, which contradicts our assumption.

\medskip

\noindent As both cases lead to absurdity, it follows that $\lambda_\Psi$ is continuous as desired.
\end{proof}

Theorem \ref{thm-below2r} has a several immediate consequences.  Recall that $G\in\cs$ is 1-generic if for every $\Sigma^0_1$ $S\subseteq \str$, there is some $\sigma\prec G$ such that either $\sigma\in S$ or for all $\tau\succeq\sigma$, $\tau\notin S$.

\begin{corollary}\label{cor-gen}
There are 1-generic sequences that are Martin-L\"of random with respect to a continuous, $\emptyset'$-computable measure.
\end{corollary}

\begin{proof}
As shown by Kautz \cite{Kau91}, every 2-random computes a 1-generic sequence.  Since no 1-generic sequence is computable, we can apply Theorem \ref{thm-below2r} to obtain the result.
\end{proof}

This provides us with the first example of a sequence that is not proper but is random with respect to a continuous, $\emptyset'$-computable measure, as Muchnik  \cite{MucSemUsp98} proved that no 1-generic sequence is random with respect to a computable measure.
 
Note further that Corollary \ref{cor-gen} is not true of all 1-generic sequences.  There are $\Delta^0_2$ 1-generic sequences, and by the remark after Lemma \ref{lem-atom}, if a $\Delta^0_2$ sequence $Y$ is random with respect to a $\emptyset'$-computable measure $\mu$, then $Y$ must be an atom of $\mu$ (and hence $\mu$ cannot be continuous).  Moreover, Corollary \ref{cor-gen} fails to hold of any 2-generic sequence (a notion obtained by relativizing the definition of 1-genericity to $\emptyset'$):  by a direct relativization of Muchnik's result mentioned in the previous paragraph, no 2-generic sequence is 2-random with respect to a $\emptyset'$-computable measure.  Using this latter fact, we obtain an alternative proof of the following result due to Nies, Stephan, and Terwijn \cite{NieSteTer05} as a Corollary of Theorem \ref{thm-below2r}.

\begin{corollary}
Every 2-random sequence forms a minimal pair with every 2-generic sequence.
\end{corollary}

\begin{proof}
Suppose there is some noncomputable $Z\in\cs$ that is computable from some 2-random sequence $X$ and from some 2-generic sequence $Y$.  By Jockusch \cite{Joc80} (who attributes the result to Martin), the collection of 2-generic sequences are downward dense, i.e., every noncomputable sequence computable from a 2-generic computes a 2-generic sequence.  So without loss of generality, we can assume that $Z$ is 2-generic (since if $Z$ is not 2-generic, it computes a 2-generic sequence that is still below both $X$ and $Y$).  Moreover, by Theorem \ref{thm-below2r}, since $Z$ is computable from a 2-random sequence, it is 2-random with respect to a $\emptyset'$-computable measure.  But this contradicts the relativization of Muchnik's theorem.
\end{proof}

The next corollary of Theorem \ref{thm-below2r} shows us that one can $\emptyset'$-computably recover unbiased 2-randomness from any noncomputable sequence computable from a 2-random sequence.  

\begin{corollary}
For every noncomputable sequence $X$ computable from some 2-random sequence, $X\oplus \emptyset'$ computes a 2-random sequence.
\end{corollary}

\begin{proof}
Let $X$ be a noncomputable sequence that is computable from some 2-random sequence.  Thus by Theorem \ref{thm-below2r}, $X$ is 2-random with respect to a $\emptyset'$-computable measure.  As shown independently by Levin \cite{ZvoLev70} and Kautz \cite{Kau91} (as well as by Schnorr and Fuchs \cite{SchFuc77}), for every sequence $Z$ that is random with respect to some computable measure, there is some Martin-L\"of random sequence $Y$ such that $Y\leq_TZ$ (in fact, $Y\equiv_T Z$).  Relativizing this result to $\emptyset'$, we get that for every sequence $Z$ that is 2-random with respect to a $\emptyset'$-computable measure, there is some 2-random sequence $Y$ such that $Y\leq_T Z\oplus \emptyset'$.  Applying this result to $X$ as given above yields the desired conclusion.
\end{proof}

We conclude this section with one last corollary of Theorem \ref{thm-below2r} and an open question.  Recall that $\mathrm{NCR}$ is the collection of sequences that are not random with respect to a continuous measure (first introduced by Reimann and Slaman in \cite{ReiSla15}).  An immediate consequence of Theorem \ref{thm-below2r} is the following.

\begin{corollary}
No 2-random sequence computes a noncomputable member of $\mathit{NCR}$.
\end{corollary}

We cannot weaken this result to hold for Demuth randomness.  We do not provide a definition of Demuth randomness here (see, for instance, \cite[Section 7.6]{DowHir10}).  For our purpose, the key fact is that the collection of 2-random sequences is a proper subset of the collection of Demuth random sequences.  
Let $Y$ be a Demuth random sequence that is not weakly 2-random. Then by a result of Hirschfeldt and Miller, $Y$ computes a noncomputable c.e.\ set $A$ (see \cite[Corollary 7.2.12]{DowHir10}).  As shown by Ku\v cera and Nies \cite{KucNie11}, such a c.e.\ set must be $K$-trivial 
(see \cite[Section 5.2]{Nie09} for a definition of $K$-triviality).  Lastly, Barmpalias, Greenberg, Montalb\'an, and Slaman \cite{BarGreMon12} showed that every $K$-trivial sequence is in $\mathrm{NCR}$.  Putting all of these pieces together, this yields a Demuth random sequence that computes a noncomputable member of $\mathrm{NCR}$.  A similar argument does not work for weak 2-randomness, as no weakly 2-random sequence computes a noncomputable $\Delta^0_2$ sequence (and every $K$-trivial sequence is $\Delta^0_2$).  We thus can ask:

\begin{question}
Can a weakly 2-random sequence compute a noncomputable member of $\mathit{NCR}$?
\end{question}

\bigskip

\section{The jump and pseudojumps of a 2-random sequence}\label{sec-operators}

In this section, we continue our study of continuous randomness by studying the behavior 
of 2-random sequences under a broader class of effective operators beyond Turing functionals.  Here we consider the Turing jump, c.e.\ operators, pseudojump operators, and operators defined in terms of pseudojump inversion.   As we will see, these too yield sequences that are continuously random.  We first consider the jump of a 2-random sequence.

\begin{theorem}\label{thm-jump}
For $X\in\cs$, if $X$ is 2-random, then $X'$ is Martin-L\"of random with respect to a continuous, $\emptyset'$-computable measure.
\end{theorem}

\begin{proof}
We proceed with a proof similar to that of Theorem \ref{thm-below2r}, with several modifications.  First, we set $S_n=\{Z\in\cs\colon \Phi_n(Z;n)\halts\}$ for each $n\in\omega$. Then as in the proof of Theorem \ref{thm-below2r}, we define a function $f\leq_T\jump$ such that for each $n\in\omega$, $f(n)$ is the least stage $s$ such that $\lambda(S_n\setminus S_{n,s})\leq 2^{-n}$.  Then we set $V_n=S_n\setminus S_{n,f(n)}$ for each $n\in\omega$, and we further set $W_n=\bigcup_{i>n}V_i$, yielding $(W_n)_{n\in\omega}$, a Martin-L\"of test relative to $\jump$.   Since $X\notin W_j$ for some $j\in\omega$, it follows that $X\notin V_i$ for all $i>j$.

Now for each $n\in\omega$, $n\in X'$ if and only if $X\in S_n$.  Moreover, for all $n>j$, $X\notin S_n\setminus S_{n,f(n)}$.  Thus, for all $n>j$ such that $X\in S_n$, we must have $X\in S_{n,f(n)}$.  We can thus conclude that for all but finitely many $n$, $n\in X'$ if and only if $X\in S_{n,f(n)}$.  Then there is some $k\in\omega$ such that for all $n\in\omega$, $n\in X'$ if and only if $X\in S_{n,f(n)+k}$. 
 
We then define a total $\jump$-computable functional $\Phi$ as follows:

\[
\Phi(Z;n)=
	\left\{
		\begin{array}{ll}
			1 & \mbox{if } Z\in S_{n,f(n)+k} \\
			0 & \mbox{otherwise.}
		\end{array}
	\right.
\]
It is immediate that $\Phi(X)=X'$.  Note further that for any 2-random sequence $Y$ and all but finitely many $n$, since $n\in Y'$ if and only if $Y\in S_{n,f(n)+k}$, it follows that $\Phi(Y)\equiv_T Y'$.

Clearly $\Phi$ is total.  Let $\lambda_\Phi$ be the $\jump$-computable measure induced by $\Phi$. Then by randomness preservation, $X'=\Phi(X)$ is Martin-L\"of random with respect to $\lambda_\Phi$.  We verify that $\lambda_\Phi$ is continuous.  Suppose otherwise, so that $\lambda_\Phi(\{A\})>0$ for some $A\in\cs$.  Then $\lambda(\{Y\in\cs\colon A\leq_T Y\oplus \jump\})>0$.  It follows from the relativization of Sacks' Theorem (due to Stillwell \cite{Sti72}) that $A\leq_T\jump$.  In addition, since $\lambda(\Phi^{-1}(\{A\}))>0$, $\Phi^{-1}(\{A\})$ must contain some 2-random sequence.  However, since no 2-random sequence computes a noncomputable $\Delta^0_2$ sequence, it follows that $A$ must be computable.  But for each 2-random sequence $Y\in\Phi^{-1}(\{A\})$, we have $A=\Phi(Y)\equiv_T Y'$, which is impossible.  Thus $\lambda_\Phi$ cannot have any atoms.

\end{proof}

Theorem \ref{thm-jump} provides additional examples of sequences that are not proper but random with respect to a continuous, $\emptyset'$-computable measure, namely the jump of every 2-random sequence.  To establish this, we just need to prove the following.

\begin{lemma}
For every $X\in\cs$, $X'$ is not proper.
\end{lemma}

\begin{proof}
We recall one definition and two facts.  First, a sequence $Z$ has diagonally noncomputable (DNC) degree if there is some $f\leq_TZ$ such that $f(n)\neq\phi_n(n)$ for all $n\in\omega$.   As shown by Ku\v cera \cite{Kuc85}, every Martin-L\"of random sequence has DNC degree.  Moreover, Arslanov's Completeness Criterion \cite{Ars81} says that  no intermediate c.e.\ set has DNC degree.  

Now, suppose that $X'$ is proper for some $X\in\cs$.  Let $C$ be a noncomputable, incomplete c.e.\ set.  As $C\leq_m\emptyset'\leq_m X'$, we have $C\leq_{tt} X'$.  By randomness preservation, the property of being proper is closed  downwards under $\mathit{tt}$-reducibility, and thus it follows that $C$ is proper.  However, every noncomputable proper sequence is Turing equivalent to a Martin-L\"of random sequence (as shown independently by Levin \cite{ZvoLev70} and Kautz \cite{Kau91}) and hence has DNC degree, which contradicts Arslanov's Completeness Criterion.
\end{proof}

Using Theorem \ref{thm-jump}, we can further consider the behavior of 2-random sequences under c.e.\ operators and pseudojump operators.  Recall that a c.e.\ operator is given by considering the domain of an oracle Turing machine with a fixed oracle; for $e\in\omega$, the $e$-th c.e.\ operator is given by the map $X\mapsto W_e^X$.  By a relativization of Posts's theorem that the halting set is 1-complete, for every $e\in\omega$ and $X\in\cs$, we have $W_e^X\leq_1 X'$.  We use this fact to prove the following.

\begin{proposition}\label{prop-pj1}
Suppose that for $e\in\omega$, we have $W_e^Y\not\leq_T\jump$ for every 2-random sequence $Y$.  If $X$ is a 2-random sequence, then $W_e^X$ is Martin-L\"of random with respect to a continuous, $\jump$-computable measure.
\end{proposition}

\begin{proof}
Fix $e\in\omega$ as in the statement of the proposition. As $W_e^X\leq_1 X'$, this 1-reduction induces a total Turing functional $\Psi$. Given a 2-random $X\in\cs$, let $\mu$ be a $\jump$-computable measure such that $X'$ is $\emptyset'$-Martin-L\"of random with respect to $\mu$ (which is guaranteed to exist by Theorem \ref{thm-jump}).  By randomness preservation, $W_e^X$ is random with respect to the induced $\emptyset'$-computable measure $\mu_\Psi$.  We claim that $\mu_\Psi$ is continuous.  Suppose not.  Then $\mu_\Psi(\{A\})>0$ for some $A\in\cs$.  Since $\mu_\Psi$ is $\jump$-computable and $A$ is a $\mu_\Psi$-atom, it follows from Lemma \ref{lem-atom} that $A$ is $\Delta^0_2$.  By Theorem \ref{thm-pres}(ii) (no randomness from nonrandomness) applied to both $\Psi$ and the $\emptyset'$-computable functional induced by taking the jump of $X$,  the set $\{Y: W_e^Y=A\}$ must contain a 2-random sequence $Z$, i.e., $W_e^Z=A$.  However, this contradicts our assumption that $W_e^Y\not\leq_T\jump$ for every 2-random sequence $Y$.  Thus $\mu_\Psi$ must be continuous.
\end{proof}

A \emph{pseudojump operator} is given by a map of the form $X\mapsto X\oplus W_e^X$ for some $e\in\omega$ (see \cite{JocSho83}, \cite{JocSho84}).  We now obtain a result similar to Proposition \ref{prop-pj1} for pseudojump operators, except that
 we do not need the additional assumption that $W_e^Y\not\leq_T\jump$ for every 2-random sequence $Y$.

\begin{proposition}\label{prop-pj2}
For every $e\in\omega$ and every 2-random sequence $X$, $X\oplus W_e^X$ is Martin-L\"of random with respect to a continuous, $\jump$-computable measure.
\end{proposition}

\begin{proof}
As in the proof of Proposition \ref{prop-pj1}, there is a total $\emptyset'$-functional $\Psi$ that maps each 2-random sequence $Y\in\cs$ to $Y\oplus W_e^Y$.   Let $\mu$ be a $\jump$-computable measure such that $X'$ is $\emptyset'$-Martin-L\"of random with respect to $\mu$. Then as we argued in the proof of Proposition \ref{prop-pj1}, given a 2-random sequence $X\in\cs$, $X\oplus W_e^X$ is $\emptyset'$-Martin-L\"of random with respect to the induced measure $\mu_\Psi$, which is $\jump$-computable.  Lastly, $\mu_\Psi$ is continuous, as $\mu$ is continuous and $\Psi$ is one-to-one.
\end{proof}

Finally, we consider pseudojump inversion.  The pseudojump inversion theorem is as follows:

\begin{theorem}[Jockusch/Shore \cite{JocSho83}]
Let $e\in\omega$.  For every $A\geq_T\jump$, there is some $B\in\cs$ such that $B\oplus W_e^B\equiv_T A$.  
\end{theorem}

Here we consider pseudojump inversion as defining an effective operator.  One can observe that the proof of the pseudojump inversion theorem (see \cite[Theorem 2.1]{JocSho83}) gives, for each $e\in\omega$, a total $\jump$-computable functional $\Xi$ such that for every $A\geq_T\jump$, setting $\Xi(A)=B$, we have $B\oplus W_e^B\equiv_T A$.  Let us review the details, which are relevant for our discussion.

Fix $e\in\omega$.  Given $A\in\cs$, we define a sequence of strings $(\tau_i)_{i\in\omega}$ such that $\tau_i\preceq\tau_{i+1}$ for all $i\in\omega$.  First we set $\tau_0=\epsilon$.  For each $i\in\omega$, given $\tau_{2i}$, we use $\jump$ to determine whether there is some $\tau\succeq\tau_{2i}$ such that $i\in W_e^\tau$.  If so, we let $\tau_{2i+1}$ be the length-lexicographically least such $\tau$; otherwise, we set $\tau_{2i+1}=\tau_{2i}$.  We then set $\tau_{2i+2}={\tau_{2i+1}}^\frown A(i)$.  Then $B=\bigcup_{i\in\omega}\tau_i$ is the desired sequence.  Based on this construction, we can prove:

\begin{theorem}
For every $e\in\omega$ and every 2-random sequence $X$, if $\Xi$ inverts the pseudojump operator with index $e$, then $\Xi(X')$ is Martin-L\"of random with respect to a continuous, $\jump$-computable measure.
\end{theorem}

\begin{proof}
Clearly the operator $\Xi$ as described above is total and $\jump$-computable.  For a 2-random $X\in\cs$, let $\mu$ be a $\jump$-computable measure such that $X'$ is $\emptyset'$-Martin-L\"of random with respect to $\mu$. Applying $\Xi$ to the jump of a 2-random sequence yields a sequence that is $\emptyset'$-Martin-L\"of random with respect to the induced $\emptyset'$-computable measure $\mu_{\Xi}$

To verify that $\mu_{\Xi}$ is continuous, we claim that $\Xi$ is one-to-one.  For $A_0,A_1\in\cs$, suppose that $\Xi(A_0)=\Xi(A_1)$.  For $j\in\{0,1\}$, let $(\tau^j_i)_{i\in\omega}$ be the sequence of strings from the above construction when applied to $A_j$.  We show by induction that $\tau^0_i=\tau^1_i$ for every $i\in\omega$, from which it follows that $A_0=A_1$.

\begin{itemize}
\item Base case:  $\tau^0_0=\epsilon=\tau^1_0$.

\item Inductive step:  For $k\in\omega$, suppose that $\tau^0_k=\tau^1_k$.  We have two cases to consider:

\medskip

\noindent Case 1:  $k=2n$ for some $n\in\omega$.  Then for $j\in\{0,1\}$, either $\tau^j_{2n+1}$ is the length-lexicographically least $\tau\succeq \tau^0_k$ such that $n\in W_e^\tau$ or $\tau^j_{2n+1}=\tau^j_{2n}$.  Either way, we have $\tau^0_{k+1}=\tau^1_{k+1}$.

\medskip

\noindent Case 2:  $k=2n+1$ for some $n\in\omega$.  Then under the assumption that $\Xi(A_0)=\Xi(A_1)$ and $\tau^0_k=\tau^1_k$, it follows that that $\tau^0_{k+1}={\tau^0_k}^\frown A_0(n)={\tau^1_k}^\frown A_1(n)=\tau^1_{k+1}$.
\end{itemize}
It follows by induction that $\tau^0_k=\tau^1_k$ for all $k\in\omega$, and hence that $A_0=A_1$.  As $\Xi$ is one-to-one, it follows that $\mu_\Xi$ is continuous.

\end{proof}

Note that, for $e\in\omega$ and $X\in\cs$, in the case that $X\oplus W_e^X$ is proper, it follows that $W_e^X$ is also proper (since $W_e^X\leq_{tt} X\oplus W_e^X$).  Thus, we can broaden our analysis of sequences that are not proper but are random with respect to a continuous, $\emptyset'$-computable measure by answering the following questions.

\begin{question}
For which $e\in\omega$ do we have that $W_e^X$ is not proper for any 2-random sequence $X$?
\end{question}

\begin{question}
For which pseudojump operators is it the case that the sequence obtained by pseudojump inversion applied to the jump of a 2-random sequence yields a nonproper sequence?
\end{question}

\section*{Acknowledgements} The research in this article was supported by 
NSA Mathematical Sciences Program Young Investigator Grant  \# H98230-16-1-0310.
Thank you to Laurent Bienvenu for early conversations that led to the main idea that motivated this study, and to the anonymous reviewers for helpful feedback that improved the presentation of the paper.

%
%
%
%
%
%
%

%
%
%
%
%
\bibliographystyle{alpha} \bibliography{continuous}

\end{document}